\documentclass[11pt,a4paper]{amsart}
\usepackage[utf8]{inputenc}

\usepackage{amsmath}
\usepackage{amsfonts}
\usepackage{amssymb}
\usepackage{amsthm}
\usepackage{enumitem}
\usepackage{cases}
\usepackage{comment}
\usepackage{esint}
\usepackage{fancyhdr}
\usepackage{float}
\usepackage[left=2cm,right=2cm]{geometry}
\usepackage{graphicx}
\usepackage{subcaption}
\usepackage{tkz-euclide}

\newtheorem{theorem}{Theorem}[section]

\newtheorem{defin}[theorem]{Definition}
\newtheorem{lemma}[theorem]{Lemma}

\newtheorem{rem}[theorem]{Remark}

\usepackage{hyperref}
\hypersetup{
  colorlinks = true, 
  urlcolor = blue, 
  linkcolor = blue, 
  citecolor = blue 
}

\let \div \relax

\let \vb \mathbf

\DeclareMathOperator{\div}{div}

\DeclareMathOperator{\Z}{\mathbb{Z}}
\DeclareMathOperator{\R}{\mathbb{R}}

\newcommand{\dd}{\mathrm{d}}
\newcommand{\e}{\varepsilon}

\newcommand{\weak}{\rightharpoonup}
\newcommand{\B}{\mathcal{B}}
\newcommand{\mass}{\mathfrak{m}}

\renewcommand{\d}{\partial}
\renewcommand{\rho}{\varrho}
\renewcommand{\phi}{\varphi}

\allowdisplaybreaks

\newcommand{\dx}{\dd{x}}
\newcommand{\dt}{\dd{t}}

\title[Homogenization and low Mach limit in critical perforation]{Homogenization and low Mach number limit of compressible Navier-Stokes equations in critically perforated domains}
\author{Peter Bella \and Florian Oschmann}
\date{}   
\begin{document}
\setlength{\parindent=0pt}
\maketitle

\begin{abstract}
In this note, we consider the homogenization of the compressible Navier-Stokes equations in a periodically perforated domain in $\R^3$. Assuming that the particle size scales like $\e^3$, where $\e>0$ is their mutual distance, and that the Mach number decreases fast enough, we show that in the limit $\e\to 0$, the velocity and density converge to a solution of the incompressible Navier-Stokes equations with Brinkman term. We strongly follow the methods of H{\"o}fer, Kowalczyk and Schwarzacher [\url{https://doi.org/10.1142/S0218202521500391}], where they proved convergence to Darcy's law for the particle size scaling like $\e^\alpha$ with $\alpha\in (1,3)$.
\end{abstract}

\section{Introduction}
We consider a bounded smooth domain $D\subset\R^3$ which for $\e>0$ is perforated by tiny obstacles of size $\e^3$, and show that solutions to the compressible Navier-Stokes equations in this domain converge as $\e\to 0$ to a solution of the incompressible Navier-Stokes equations with Brinkman term. To the best of our knowledge, this is the first result of homogenization of compressible fluids for a critically sized perforation.\\

There is a vast of literature concerning the homogenization of fluid flows in perforated domains. We will just cite a few. For incompressible fluids, Allaire found in \cite{Allaire90a} and \cite{Allaire90b} that, concerning the ratios of particle size and distance, there are mainly three regimes of particle sizes $\e^\alpha$, where $\alpha\geq 1$. Heuristically, if the particles are large, the velocity will slow down and finally stop. This phenomenon occurs if (in three dimensions) $\alpha\in [1,3)$ and gives rise to Darcy's law. When the particles are very small, i.e., $\alpha>3$, they should not affect the fluid, yielding that in the limit, the fluid motion is still governed by the Stokes or Navier-Stokes equations. The third regime is the so-called critical case $\alpha=3$, where the particles are large enough to put some friction on the fluid, but not too large to stop the flow. For incompressible fluids, the non-critical cases $\alpha\in (1,3)$ and $\alpha>3$ were considered in \cite{Allaire90b}, while \cite{Allaire90a} dealt with the critical case $\alpha=3$. The case $\alpha=1$ was treated in \cite{Allaire89}. In all the aforementioned literature, the proofs were given by means of suitable oscillating test functions, first introduced by Tartar in \cite{Tartar1980} and later adopted by Cioranescu and Murat in \cite{CioranescuMurat82} for the Poisson equation.\\

In the critical case, the additional friction term is the main part of Brinkman's law. Cioranescu and Murat considered in \cite{CioranescuMurat82} the Poisson equation in a perforated domain, where they found in the limit ``a strange term coming from nowhere''. This Brinkman term purely comes from the presence of holes in the domain $D_\e$. It physically represents the energy of boundary layers around each obstacle, as its columns are proportional to the drag force around a single particle \cite[Proposition~2.1.4 and Remark~2.1.5]{Allaire90a}.\\

The assumptions on the distribution of the holes can also be generalized. For the critical case, Giunti, H{\"o}fer, and Vel{\'a}zquez considered in \cite{GiuntiHoeferVelazquez18} homogenization of the Poisson equation in a randomly perforated domain. They showed that the ``strange term'' also occurs in their setting. Hillairet considered in \cite{Hillairet18} the Stokes equations and random obstacles with a hard sphere condition. This condition was removed by Giunti and H{\"o}fer \cite{GiuntiHoefer19}, where they showed that for incompressible fluids and randomly distributed holes with random radii, the randomness does not affect the convergence to Brinkman's law. More recently,  for large particles, Giunti showed in \cite{Giunti21} a similar convergence result to Darcy's law.\\

Unlike as for incompressible fluids, the homogenization theory for compressible fluids is rather sparse. Masmoudi considered in \cite{Masmoudi02} the case $\alpha=1$ of large particles, giving rise to Darcy's law. For large particles with $\alpha\in (1,3)$, Darcy's law was just recently treated in \cite{SchwarzacherDarcy} for a low Mach number limit. The case of small particles ($\alpha>3$) was treated in \cite{DieningFeireislLu,FeireislLu,LuSchwarzacher} for different growing conditions on the pressure. Random perforations in the spirit of \cite{GiuntiHoefer19} for small particles were considered by the authors in \cite{BellaOschmann2021a}, where in the limit, the equations remain unchanged as in the periodic case.\\

We want to emphasize that the methods presented here are strongly related to those of \cite{SchwarzacherDarcy}. As a matter of fact, their techniques used in the case of large holes also apply in our case for holes having critical size.\\

{\bf Notation:} Throughout the whole paper, we denote the Frobenius scalar product of two matrices $A,B\in\R^{3\times 3}$ by $A:B:=\sum_{1\leq i,j\leq 3} A_{ij}B_{ij}$. Further, we use the standard notation for Lebesgue and Sobolev spaces, where we denote this spaces even for vector valued functions as in scalar case, e.g., $L^p(D)$ instead of $L^p(D;\R^3)$. Moreover, $C>0$ denotes a constant which is independent of $\e$ and might change its value whenever it occurs.\\

{\bf Organization of the paper:} The paper is organized as follows:\\
In Section \ref{sec:Main}, we give a precise definition of the perforated domain $D_\e$ and state our main results for the steady Navier-Stokes equations. In Section \ref{sec:TestFunctions}, we introduce oscillating test functions, which will be crucial to show convergence of the velocity, density, and pressure. Section \ref{sec:BogAndUnifBounds} is devoted to invoke the concept of Bogovski\u{\i}'s operator as an inverse of the divergence, which is used to give uniform bounds independent of $\e$. In Section \ref{sec:convergence}, we show how to pass to the limit $\e\to 0$ and obtain the limiting equations.

\section{Setting and main results}\label{sec:Main}
Consider a bounded domain $D\subset \R^3$ with smooth boundary. Let $\e>0$ and cover $D$ with a regular mesh of size $2\e$. Set $x_i^\e\in (2\e\Z)^3$ as the center of the cell with index $i$ and $P_i^\e:=x_i^\e+(-\e,\e)^3$. Further, let $T\Subset B_1(0)$ be a compact and simply connected set with smooth boundary and set $T_i^\e:=x_i^\e+\e^3 T$. We now define the perforated domain as
\begin{align}\label{def:D_eps}
D_\e:=D\setminus\bigcup_{i\in K_\e} T_i^\e, \quad K_\e:=\{i:\overline{P_i^\e}\subset D\}.
\end{align}
By the periodic distribution of the holes, the number of holes inside $D_\e$ satisfy
\begin{align*}
|K_\e|\leq C\, \frac{|D|}{\e^3} \quad \text{for some $C>0$ independent of $\e$}.
\end{align*}

In $D_\e$, we consider the steady compressible Navier-Stokes equations
\begin{align}\label{eq:NSE}
\begin{cases}
\div(\rho_\e \vb u_\e\otimes \vb u_\e)-\div\mathbb{S}(\nabla\vb u_\e)+\frac{1}{\e^\beta}\nabla \rho_\e^\gamma=\rho_\e \vb f+\vb g & \text{in } D_\e,\\
\div(\rho_\e \vb u_\e)=0 & \text{in } D_\e,\\
\vb u_\e=0 & \text{on } \d D_\e,\\
\end{cases}
\end{align}

where $\rho_\e,$ $\vb u_\e$ are the fluids density and velocity, respectively, and $\mathbb{S}(\nabla \vb u_\e)$ is the Newtonian viscous stress tensor of the form
\begin{align}\label{eq:stressTensor}
\mathbb{S}(\nabla \vb u)=\mu \bigg(\nabla \vb u+\nabla^T \vb u-\frac23 \div(\vb u)\mathbb{I}\bigg)+\eta\div(\vb u)\mathbb{I}
\end{align}

with viscosity coefficients $\mu>0$, $\eta\geq 0$. Further, we assume that $\gamma\geq 3$, $\beta>3\,(\gamma+1)$, and $\vb f, \vb g\in L^\infty(D)$ are given. Since the equations \eqref{eq:NSE} are invariant under adding a constant to the pressure term $\e^{-\beta}\rho_\e^\gamma$, we define
\begin{align}\label{def:pressure}
p_\e:=\e^{-\beta} (\rho_\e^\gamma-\langle \rho_\e^\gamma\rangle_\e),
\end{align}

where $\langle\cdot\rangle_\e$ denotes the mean value over $D_\e$, given by
\begin{align*}
\langle f\rangle_\e=\frac{1}{|D_\e|}\int_{D_\e} f\,\dx.
\end{align*}

We will show convergence of the velocity $\vb u_\e$ and the pressure $p_\e$ to limiting functions $\vb u$ and $p$, respectively, such that the couple $(\vb u,p)$ solves the incompressible steady Navier-Stokes-Brinkman equations
\begin{align*}
\begin{cases}
\div(\rho_0 \vb u\otimes \vb u)-\mu\Delta\vb u+\nabla p+\mu M\vb u=\rho_0 \vb f+\vb g& \text{in } D,\\
\div(\vb u)=0 & \text{in } D,\\
\vb u=0 & \text{on } \d D,
\end{cases}
\end{align*}

where the resistance matrix $M$ is introduced in the next section, and the constant $\rho_0$ is the strong limit of $\rho_\e$ in $L^{2\gamma}(D)$, which is determined by the mass constraint on $\rho_\e$ as formulated in Definition~\ref{def:WeakSol} below.\\

Before stating our main result, we introduce the standard concept of finite energy weak solutions to~\eqref{eq:NSE}.
\begin{defin}\label{def:WeakSol}
Let $D_\e$ be as in \eqref{def:D_eps} and $\gamma\geq 3$, $\mass>0$ be fixed. We say a couple $(\rho_\e, \vb u_\e)$ is a \emph{finite energy weak solution} to system \eqref{eq:NSE} if
\begin{gather*}
\rho_\e\in L^{2\gamma}(D_\e),\quad \vb u_\e\in W^{1,2}_0(D_\e),\\
\rho_\e\geq 0 \text{ a.e.~in $D_\e$,}\quad \int_{D_\e}\rho_\e\,\dx=\mass,\\
\int_{D_\e} \rho_\e \vb u_\e\cdot\nabla\psi\,\dx=0,\\
\int_{D_\e} p_\e\div \phi+(\rho_\e \vb u_\e\otimes \vb u_\e):\nabla\phi-\mathbb{S}(\nabla \vb u_\e):\nabla\phi+(\rho_\e \vb f+\vb g)\cdot\phi\,\dx=0
\end{gather*}
for all test functions $\psi\in C_c^\infty(D_\e)$ and all test functions $\phi\in C_c^\infty(D_\e;\R^3)$, where $p_\e$ is given in \eqref{def:pressure}, and the energy inequality
\begin{align}\label{eq:energyineq}
\int_{D_\e} \mathbb{S}(\nabla \vb u_\e):\nabla \vb u_\e\,\dx\leq \int_{D_\e} (\rho_\e \vb f+\vb g)\cdot \vb u_\e\,\dx
\end{align}

holds.
\end{defin}

\begin{rem}
Existence of finite energy weak solutions to system \eqref{eq:NSE} is known for all values $\gamma>3/2$; see, for instance, \cite[Theorem~4.3]{Novotny2004}. However, we need the assumption $\gamma\geq 3$ to bound the convective term $\div(\rho_\e\vb u_\e\otimes\vb u_\e)$ in a proper way, see Section \ref{sec:BogAndUnifBounds}.
\end{rem}

Let us denote the zero extension of a function $f$ with $D_\e$ as its domain of definition by $\tilde{f}$, that is,
\begin{align*}
\tilde{f}=f \text { in } D_\e,\quad \tilde{f}=0 \text{ in } \R^3\setminus D_\e.
\end{align*}

Our main result for the stationary Navier-Stokes equations now reads as follows:
\begin{theorem}\label{thm:Main}
Let $D\subset\R^3$ be a bounded domain with smooth boundary, $0<\e<1$, $D_\e$ be as in \eqref{def:D_eps}, $\gamma\geq 3$, $\mass>0$ and $\vb f,\vb g\in L^\infty(D)$. Let $\beta>3\,(\gamma+1)$ and $(\rho_\e,\vb u_\e)$ be a sequence of finite energy weak solutions to problem \eqref{eq:NSE}. Then, with $p_\e$ defined in \eqref{def:pressure}, we can extract subsequences (not relabeled) such that
\begin{align*}
\tilde{\rho}_\e &\to \rho_0 \hspace{.6em} \text{ strongly in } L^{2\gamma}(D),\\
\tilde{p}_\e&\weak p \quad \text{ weakly in } L^2(D),\\
\tilde{\vb u}_\e &\weak \vb u \hspace{.9em} \text{ weakly in } W_0^{1,2}(D),
\end{align*}

where $\rho_0=\mass/|D|$ is constant and $(p,\vb u)\in L^2(D)\times W_0^{1,2}(D)$ with $\int_D p=0$ is a weak solution to the steady incompressible Navier-Stokes-Brinkman equations
\begin{align}\label{eq:Brinkman}
\begin{cases}
\div(\rho_0\vb u\otimes \vb u)+\nabla p -\mu\Delta\vb u+\mu M\vb u=\rho_0\vb f+\vb g& \text{in } D,\\
\div(\vb u)=0 & \text{in } D,\\
\vb u=0 & \text{on } \d D,
\end{cases}
\end{align}

where $M$ will be defined in \eqref{def:Resistance}.
\end{theorem}

\begin{rem}
It it well known that the solution to system \eqref{eq:Brinkman} is unique if $\vb f$ and $\vb g$ are ``sufficiently small'', see, e.g., \cite[Chapter II, Theorem~1.3]{Temam1977navier}. This smallness assumption can be dropped in the case of Stokes equations, i.e., without the convective term $\div(\rho_0 \vb u\otimes \vb u)$.
\end{rem}

\section{The cell problem and oscillating test functions}\label{sec:TestFunctions}
In this section, we introduce oscillating test functions and define the resistance matrix $M$, following the original work of Allaire \cite{Allaire90a}. We repeat here the definition of these functions as well as the estimates given in \cite{SchwarzacherDarcy}.\\

Consider for a single particle $T$ the solution $(q_k,\vb w_k)$ to the cell problem
\begin{align}\label{eq:ModelProblem}
\begin{cases}
\nabla q_k-\Delta \vb w_k=0 & \text{in } \R^3\setminus T,\\
\div(\vb w_k)=0 & \text{in } \R^3\setminus T,\\
\vb w_k=0 & \text{on } \d T,\\
\vb w_k=\vb e_k & \text{at infinity},
\end{cases}
\end{align}

where $\vb e_k$ is the $k$-th unit basis vector of the canonical basis of $\R^3$. Note that the solution exists and is unique, see, e.g., \cite[Chapter~V]{Galdi2011}. Let us further recall the definition of oscillating test functions as made in \cite{Allaire90a} (see also \cite{SchwarzacherDarcy}):\\

We set
\begin{align*}
\vb w_k^\e=\vb e_k,\quad q_k^\e=0\text{ in } P_i^\e\cap D
\end{align*}

for each $P_i^\e$ with $P_i^\e\cap\d D\neq\emptyset$. Now, we denote $B_i^r:=B_r(x_i^\e)$ and split each cell $P_i^\e$ entirely included in $D$ into the following four parts:

\begin{align*}
\overline{P_i^\e}=T_i^\e\cup \overline{C_i^\e}\cup \overline{D_i^\e}\cup \overline{K_i^\e},
\end{align*}

where $C_i^\e$ is the open ball centered at $x_i^\e$ with radius $\e/2$ and perforated by the hole $T_i^\e$, $D_i^\e=B_i^\e\setminus B_i^{\e/2}$ is the ball with radius $\e$ perforated by the ball with radius $\e/2$, and $K_i^\e=P_i^\e\setminus B_i^\e$ are the remaining corners, see Figure~\ref{fig:SplittingOfCell}. 

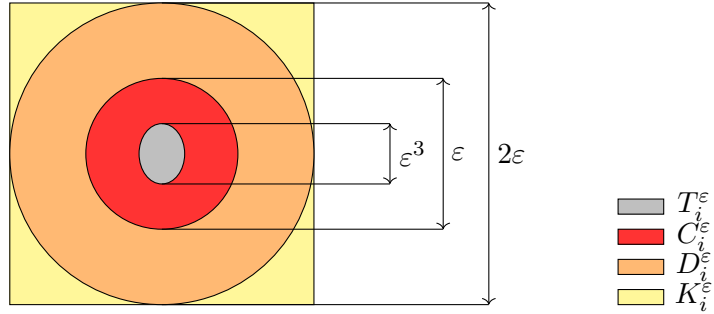
\begin{figure}[H]
\begin{tikzpicture}
\draw[fill=yellow!50] (-2,-2) rectangle (2,2);
\draw[fill=orange!55] (0,0) circle (2cm);
\draw[fill=red!80] (0,0) circle (1cm);
\draw[fill=lightgray] (0,0) ellipse (.3cm and .4cm);
\draw (0,.4)--(3,.4);
\draw (0,-.4)--(3,-.4);
\draw[<->] (3,-.4)--(3,.4);
\node at (3.3,0) {$\e^3$};
\draw (0,1)--(3.7,1);
\draw (0,-1)--(3.7,-1);
\draw[<->] (3.7,-1)--(3.7,1);
\node at (3.9,0) {$\e$};
\draw (0,2)--(4.3,2);
\draw (0,-2)--(4.3,-2);
\draw[<->] (4.3,-2)--(4.3,2);
\node at (4.6,0) {$2\e$};


\draw[fill=yellow!50] (6,-2) rectangle (6.6,-1.8);
\node at (7,-1.9) {$K_i^\e$};
\draw[fill=orange!55] (6,-1.6) rectangle (6.6,-1.4);
\node at (7,-1.5) {$D_i^\e$};
\draw[fill=red!80] (6,-1.2) rectangle (6.6,-1);
\node at (7,.-1.1) {$C_i^\e$};
\draw[fill=lightgray] (6,-0.8) rectangle (6.6,-0.6);
\node at (7,-0.7) {$T_i^\e$};
\end{tikzpicture}
\caption{Splitting of the cell $P_i^\e$}
\label{fig:SplittingOfCell}
\end{figure}

In these parts, we define
\begin{align*}
&\begin{cases}
\vb w_k^\e=\vb e_k\\
q_k^\e=0
\end{cases} \text{ in } K_i^\e,
&&\begin{cases}
\nabla q_k^\e-\Delta \vb w_k^\e=0\\
\div(\vb w_k^\e)=0
\end{cases} \text{ in } D_i^\e,\\
&\begin{cases}
\vb w_k^\e(x)=\vb w_k\big(\frac{x}{\e^3}\big)\\
q_k^\e(x)=\frac{1}{\e^3}q_k\big(\frac{x}{\e^3}\big)
\end{cases} \text{ in } C_i^\e,
&&\begin{cases}
\vb w_k^\e=0\\
q_k^\e=0
\end{cases} \text{ in } T_i^\e,
\end{align*}

where we impose matching Dirichlet boundary conditions and $(q_k,\vb w_k)$ is the solution to the cell problem \eqref{eq:ModelProblem}. As shown in \cite[Lemma~3.5]{SchwarzacherDarcy}, we have for the functions $(q_k^\e,\vb w_k^\e)$ and all $p>\frac32$ the estimates
\begin{align}
\|\nabla \vb w_k^\e\|_{L^p(D)}+\|q_k^\e\|_{L^p(D)}&\leq C\e^{3\big(\frac2p-1\big)},\label{eq:estimateLp_wk}\\
\|\nabla q_k^\e\|_{L^p(\cup_i C_i^\e)}&\leq C\e^{6\big(\frac1p-1\big)},\label{eq:estimateNabla_qk}\\
\|\nabla \vb w_k^\e\|_{L^2(\cup_i B_i^\e\setminus B_i^{\e/4})}+\|q_k^\e\|_{L^2(\cup_i B_i^\e\setminus B_i^{\e/4})}&\leq C\e,\label{eq:estimateL2_wk}
\end{align}

where the constant $C>0$ does not depend on $\e$. Moreover, we have the following Theorem due to Allaire:
\begin{theorem}[{\cite[page 214, Proposition~1.1.2 and Lemma~2.3.6]{Allaire90a}}]\label{thm:propertiesCorrectors}~ \\
The functions $(q_k^\e,\vb w_k^\e)$ fulfill:
\begin{itemize}
\item[(H1)] $q_k^\e\in L^2(D),\quad \vb w_k^\e\in W^{1,2}(D)$;
\item[(H2)] $\div\vb w_k^\e=0$ in $D$ and $\vb w_k^\e=0$ on the holes $T_i^\e$;
\item[(H3)] $\vb w_k^\e\weak \vb e_k$ in $W^{1,2}(D),\quad q_k^\e\weak 0$ in $L^2(D)/\R$;
\item[(H4)] For any $\nu_\e,\nu\in W^{1,2}(D)$ with $\nu_\e=0$ on the holes $T_i^\e$ and $\nu_\e\weak\nu$, and any $\phi\in\mathcal{D}(D)$, we have
\begin{align*}
\langle\nabla q_k^\e-\Delta\vb w_k^\e, \phi\nu_\e\rangle_{W^{-1,2}(D), W_0^{1,2}(D)}\to \langle M\vb e_k, \phi\nu\rangle_{W^{-1,2}(D), W_0^{1,2}(D)},
\end{align*}

where the resistance matrix $M\in W^{-1,\infty}(D)$ is defined by its entries $M_{ik}$ via
\begin{align}\label{def:Resistance}
\langle M_{ik}, \phi\rangle_{\mathcal{D}^\prime(D), \mathcal{D}(D)}=\lim_{\e\to 0}\int_D\phi \nabla \vb w_i^\e:\nabla\vb w_k^\e\,\dx
\end{align}

for any test function $\phi\in\mathcal{D}(D)$.
\end{itemize}

Further, for any $p\geq 1$,
\begin{align*}
\|\vb w_k^\e-\vb e_k\|_{L^p(D)}\to 0.
\end{align*}
\end{theorem}

\begin{rem}
This definition of $M$ yields that the matrix is symmetric and positive definite in the sense that for all test functions $\phi_i\in\mathcal{D}(D)$ and $\Phi=(\phi_i)_{1\leq i\leq 3}$,
\begin{align*}
\langle M\Phi,\Phi\rangle_{\mathcal{D}^\prime(D),\mathcal{D}(D)}=\lim_{\e\to 0} \int_D \bigg|\sum_{i=1}^3 \phi_i\nabla\vb w_i^\e\bigg|^2\,\dx\geq 0,
\end{align*}

thus implying that there exists at least one solution to system \eqref{eq:Brinkman}.
\end{rem}

\section{Bogovski\u{\i}'s operator and uniform bounds for the steady Navier-Stokes equations}\label{sec:BogAndUnifBounds}

As in \cite{DieningFeireislLu}, we have the following result for the inverse of the divergence operator:
\begin{theorem}[{\cite[Theorem~2.3]{DieningFeireislLu}}]\label{thm:Bog}
Let $1<q<\infty$ and $D_\e$ be defined as in \eqref{def:D_eps}. There exists a bounded linear operator
\begin{align*}
\B_\e:\bigg\{f\in L^q(D_\e):\int_{D_\e} f\,\dx=0\bigg\}\to W_0^{1,q}(D_\e)
\end{align*}

such that for any $f\in L^q(D_\e)$ with $\int_{D_\e} f\,\dx=0$,
\begin{align*}
\div\B_\e(f)=f\text{ in } D_\e,\quad \|\B_\e(f)\|_{W_0^{1,q}(D_\e)}\leq C\bigg(1+\e^{3\big(\frac2q-1\big)}\bigg)\|f\|_{L^q(D_\e)},
\end{align*}

where the constant $C>0$ does not depend on $\e$.
\end{theorem}

We will use this result to bound the pressure $p_\e$ by the density $\rho_\e$. Since the main ideas how to get uniform bounds on $\vb u_\e$, $\rho_\e$, and $p_\e$ are given in \cite{SchwarzacherDarcy}, we just sketch the proof in our case. First, by Korn's inequality and \eqref{eq:energyineq}, we find
\begin{align*}
\mu\|\nabla\vb u_\e\|_{L^2(D_\e)}^2\leq \|\rho_\e\|_{L^\frac65 (D_\e)}\|\vb u_\e\|_{L^6(D_\e)} \|\vb f\|_{L^\infty(D)}+\|\vb g\|_{L^\infty(D)}\|\vb u_\e\|_{L^1(D_\e)}.
\end{align*}

Together with Sobolev embedding, we obtain
\begin{align*}
\|\vb u_\e\|_{L^6(D_\e)}\leq C\|\nabla\vb u_\e\|_{L^2(D_\e)},
\end{align*}

which yields
\begin{align}\label{eq:inequVelDensity}
\|\vb u_\e\|_{L^6(D_\e)}+\|\nabla\vb u_\e\|_{L^2(D_\e)}\leq C(\|\rho_\e\|_{L^\frac65 (D_\e)}+1).
\end{align}

To get uniform bounds on the velocity, we first have to estimate the density. To this end, let $\B_\e$ be as in Theorem~\ref{thm:Bog}. Testing the first equation in \eqref{eq:NSE} with $\B_\e(p_\e)\in W_0^{1,2}(D_\e)$ yields
\begin{align*}
\|p_\e\|_{L^2(D_\e)}^2&=\int_{D_\e} p_\e\div\B_\e(p_\e)\,\dx\\
&=\int_{D_\e}\mathbb{S}(\nabla \vb u_\e):\nabla\B_\e(p_\e)-(\rho_\e\vb u_\e\otimes \vb u_\e):\nabla\B_\e(p_\e)-(\rho_\e\vb f+\vb g)\cdot\B_\e(p_\e)\,\dx.
\end{align*}

Recalling $\rho_\e\in L^{2\gamma}(D_\e)$ and $\gamma\geq 3$, this leads to
\begin{align*}
\|p_\e\|_{L^2(D_\e)}^2&\overset{\phantom{\eqref{eq:inequVelDensity}}}{\leq} C(\|\nabla \vb u_\e\|_{L^2(D_\e)}+\|\rho_\e\|_{L^6(D_\e)}\|\vb u_\e\|_{L^6(D_\e)}^2) \|\nabla\B_\e(p_\e)\|_{L^2(D_\e)}\\
&\quad +C\big(\|\vb f\|_{L^\infty(D_\e)} \|\rho_\e\|_{L^{2\gamma}(D_\e)}+\|\vb g\|_{L^\infty(D_\e)}\big)\|\B_\e(p_\e)\|_{L^2(D_\e)}\\
&\overset{\eqref{eq:inequVelDensity}}{\leq} C (\|\rho_\e\|_{L^\frac65(D_\e)}+1+\|\rho_\e\|_{L^6(D_\e)}(\|\rho_\e\|_{L^\frac65(D_\e)}^2+1))\|\nabla\B_\e(p_\e)\|_{L^2(D_\e)}\\
&\quad +C(\|\rho_\e\|_{L^{2\gamma}(D_\e)}+1)\|\B_\e(p_\e)\|_{L^2(D_\e)}\\
&\overset{\phantom{\eqref{eq:inequVelDensity}}}{\leq} C(\|\rho_\e\|_{L^{2\gamma}(D_\e)}+\|\rho_\e\|_{L^6(D_\e)} \|\rho_\e\|_{L^\frac65(D_\e)}^2+1)\|\B_\e(p_\e)\|_{W_0^{1,2}(D_\e)}\\
&\overset{\phantom{\eqref{eq:inequVelDensity}}}{\leq} C (\|\rho_\e\|_{L^{2\gamma}(D_\e)}+\|\rho_\e\|_{L^{2\gamma}(D_\e)}^3+1) \|\B_\e(p_\e)\|_{W_0^{1,2}(D_\e)}\\
&\overset{\phantom{\eqref{eq:inequVelDensity}}}{\leq} C(\|\rho_\e\|_{L^{2\gamma}(D_\e)}+\|\rho_\e\|_{L^{2\gamma}(D_\e)}^3+1)\|p_\e\|_{L^2(D_\e)},
\end{align*}

that is,
\begin{align}\label{eq:estimatePressure}
\|p_\e\|_{L^2(D_\e)}\leq C(\|\rho_\e\|_{L^{2\gamma}(D_\e)}+\|\rho_\e\|_{L^{2\gamma}(D_\e)}^3+1).
\end{align}

Further, we have
\begin{align*}
\langle\rho_\e\rangle_\e=\frac{1}{|D_\e|}\int_{D_\e} \rho_\e\,\dx=\frac{\mass}{|D_\e|}
\end{align*}

and
\begin{align*}
\frac{1}{\e^\beta}\|\rho_\e^\gamma-\langle\rho_\e\rangle_\e^\gamma\|_{L^2(D_\e)}\leq \frac{C}{\e^\beta}\|\rho_\e^\gamma-\langle\rho_\e^\gamma\rangle_\e\|_{L^2(D_\e)} \overset{\eqref{def:pressure}}{=} C\|p_\e\|_{L^2(D_\e)},
\end{align*}

see \cite[Section 3.3 and inequality (4.7)]{SchwarzacherDarcy}. This yields
\begin{align*}
&\frac{1}{\e^\beta}\|\rho_\e^\gamma-\langle\rho_\e\rangle_\e^\gamma\|_{L^2(D_\e)} \leq C\|p_\e\|_{L^2(D_\e)}\leq C(\|\rho_\e\|_{L^{2\gamma}(D_\e)}+\|\rho_\e\|_{L^{2\gamma}(D_\e)}^3+1)\\
&\leq C \bigg(\|\rho_\e^\gamma-\langle\rho_\e\rangle_\e^\gamma\|_{L^2(D_\e)}^\frac{1}{\gamma}+\frac{\mass}{|D_\e|^{1-1/(2\gamma)}}+\|\rho_\e^\gamma-\langle\rho_\e\rangle_\e^\gamma\|_{L^2(D_\e)}^\frac{3}{\gamma}+\frac{\mass^3}{|D_\e|^{3-3/(2\gamma)}}+1\bigg).
\end{align*}

Together with
\begin{align*}
ab^\frac1p\leq b+a^{p^\prime}\quad \forall a,b>0,\; \frac1p+\frac{1}{p^\prime}=1,
\end{align*}

which is a consequence of Young's inequality, we obtain, using $\gamma\geq 3$ and the fact that we may assume $\e\leq 1$ small enough,
\begin{align*}
\frac{1}{\e^\beta}\|\rho_\e^\gamma-\langle\rho_\e\rangle_\e^\gamma\|_{L^2(D_\e)}&\leq \frac{1}{4\e^\beta}\|\rho_\e^\gamma-\langle\rho_\e\rangle_\e^\gamma\|_{L^2(D_\e)}+C+\frac{1}{4\e^\beta}\|\rho_\e^\gamma-\langle\rho_\e\rangle_\e^\gamma\|_{L^2(D_\e)}+C^\prime\\
&=\frac{1}{2\e^\beta}\|\rho_\e^\gamma-\langle\rho_\e\rangle_\e^\gamma\|_{L^2(D_\e)}+C.
\end{align*}

Using that $|\rho_\e-\langle\rho_\e\rangle_\e|^\gamma\leq |\rho_\e^\gamma-\langle\rho_\e\rangle_\e^\gamma|$, which is a consequence of the triangle inequality for the metric $d(a,b)=|a-b|^\frac{1}{\gamma}$ for $\gamma\geq 1$, we conclude
\begin{align*}
\frac{1}{\e^\beta}\|\rho_\e-\langle\rho_\e\rangle_\e\|_{L^{2\gamma}(D_\e)}^\gamma\leq \frac{1}{\e^\beta}\|\rho_\e^\gamma-\langle\rho_\e\rangle_\e^\gamma\|_{L^2(D_\e)}\leq C,
\end{align*}

which further gives rise to
\begin{align*}
\|\rho_\e\|_{L^{2\gamma}(D_\e)}\leq \|\rho_\e-\langle\rho_\e\rangle_\e\|_{L^{2\gamma}(D_\e)}+C\langle\rho_\e\rangle_\e\leq C.
\end{align*}

In view of \eqref{eq:inequVelDensity} and \eqref{eq:estimatePressure}, we finally establish
\begin{align}\label{eq:unifBounds}
\begin{split}
\|\vb u_\e\|_{W_0^{1,2}(D_\e)}&\leq C,\\
\|\rho_\e\|_{L^{2\gamma}(D_\e)}&\leq C,\\
\|p_\e\|_{L^2(D_\e)}&\leq C,\\
\|\rho_\e-\langle\rho_\e\rangle_\e\|_{L^{2\gamma}(D_\e)}&\leq C\e^\frac{\beta}{\gamma}
\end{split}
\end{align}

for some constant $C>0$ independent of $\e$.

\section{Convergence proof for the steady case}\label{sec:convergence}
The proof of convergence we give here is essentially the same as in \cite{SchwarzacherDarcy}. We thus just sketch the steps done there while highlighting the differences.\\

\begin{proof}[Proof of Theorem~\ref{thm:Main}]
\emph{Step 1:} Recall that, for a function $f$ defined on $D_\e$, we denote by $\tilde{f}$ its zero prolongation to $\R^3$. By the uniform estimates \eqref{eq:unifBounds}, we can extract subsequences (not relabeled) such that
\begin{align*}
&\tilde{\vb u}_\e\weak \vb u\text{ weakly in } W_0^{1,2}(D),\\
&\tilde{p}_\e\weak p\text{ weakly in } L^2(D),\\
&\tilde{\rho}_\e\to \rho_0 \text{ strongly in } L^{2\gamma}(D),
\end{align*}

where $\rho_0=\mass/|D|>0$ is constant. The strong convergence of the density is obtained by
\begin{align*}
\|\tilde{\rho}_\e-\rho_0\|_{L^{2\gamma}(D)}&\leq \|\rho_0\|_{L^{2\gamma}(D\setminus D_\e)}+\|\rho_\e-\langle\rho_\e\rangle_\e\|_{L^{2\gamma}(D_\e)}+\|\langle\rho_\e\rangle_\e-\rho_0\|_{L^{2\gamma}(D_\e)}\\
&\leq \rho_0 |D\setminus D_\e|^{\frac{1}{2\gamma}}+C\e^\frac{\beta}{\gamma}+\mass|D_\e|^\frac{1}{2\gamma}\bigg(\frac{1}{|D_\e|}-\frac{1}{|D|}\bigg)\to 0,
\end{align*}

since $|D_\e|\to |D|$. Due to Rellich's theorem, we further have
\begin{align*}
\tilde{\vb u}_\e\to \vb u \text{ strongly in } L^q(D) \text{ for all } 1\leq q<6.
\end{align*}

\emph{Step 2:} We begin by proving that the limiting velocity $\vb u$ is solenoidal. To this end, let $\phi\in\mathcal{D}(\R^3)$. By the second equation of \eqref{eq:NSE}, we have
\begin{align*}
0=\int_{\R^3} \tilde{\rho}_\e\tilde{\vb u}_\e\cdot\nabla\phi\,\dx \to \rho_0\int_D \vb u\cdot\nabla\phi\,\dx.
\end{align*}

This together with the compactness of the trace operator yields
\begin{align}\label{eq:solenoidal}
\begin{cases}
\div\vb u=0 & \text{ in } D,\\
\vb u=0 & \text{ on } \d D.
\end{cases}
\end{align}

\emph{Step 3:} To prove convergence of the momentum equation, let $\phi\in\mathcal{D}(D)$ and use $\phi \vb w_k^\e$ as test function in the first equation of \eqref{eq:NSE}. This yields
\begin{align*}
\int_D \mathbb{S}(\nabla\tilde{\vb u}_\e):\nabla (\phi \vb w_k^\e)\dx=\!\int_D (\tilde{\rho}_\e\tilde{\vb u}_\e\otimes\tilde{\vb u}_\e):\nabla(\phi\vb w_k^\e)\dx+\int_{D}\tilde{p}_\e \div(\phi \vb w_k^\e)\dx+\int_D (\tilde{\rho}_\e \vb f+\vb g)\cdot (\phi \vb w_k^\e)\dx.
\end{align*}

Using the definition of $\mathbb{S}$ in \eqref{eq:stressTensor} and the fact that $\div(\vb w_k^\e)=0$ by (H2) of Theorem~\ref{thm:propertiesCorrectors}, we rewrite the left hand side as
\begin{align*}
&\int_D \mathbb{S}(\nabla\tilde{\vb u}_\e):\nabla(\phi \vb w_k^\e)\,\dx=\mu \int_D \nabla\tilde{\vb u}_\e:\nabla(\phi\vb w_k^\e)\,\dx+\bigg(\frac{\mu}{3}+\eta\bigg)\int_D \div(\tilde{\vb u}_\e)\div(\phi\vb w_k^\e)\,\dx\\
&=\mu \int_D \nabla\vb w_k^\e:\nabla(\phi\tilde{\vb u}_\e)+\nabla \tilde{\vb u}_\e:(\vb w_k^\e\otimes\nabla\phi)-\nabla \vb w_k^\e:(\tilde{\vb u}_\e\otimes\nabla\phi)\,\dx+\bigg(\frac{\mu}{3}+\eta\bigg)\int_D \div(\tilde{\vb u}_\e) \vb w_k^\e\cdot\nabla\phi\,\dx
\end{align*}

and add the term $-\int_D q_k^\e\div(\phi\tilde{\vb u}_\e)\,\dx$ to both sides to obtain
\begin{align*}
&\underbrace{\mu\int_D \! \nabla \vb w_k^\e:\nabla(\phi\tilde{\vb u}_\e)-q_k^\e\div(\phi\tilde{\vb u}_\e)\,\dx}_{I_1}\\
&+\underbrace{\mu\int_D \! \nabla\tilde{\vb u}_\e:(\vb w_k^\e\otimes\nabla\phi)-\nabla \vb w_k^\e:(\tilde{\vb u}_\e\otimes\nabla\phi)\,\dx}_{I_2}+\underbrace{\bigg(\frac{\mu}{3}+\eta\bigg)\int_D \! \div(\tilde{\vb u}_\e) \vb w_k^\e\cdot\nabla\phi\,\dx}_{I_3}\\
&=\underbrace{\int_D (\tilde{\rho}_\e\tilde{\vb u}_\e\otimes\tilde{\vb u}_\e):\nabla(\phi\vb w_k^\e)\,\dx}_{I_4}+\underbrace{\int_D \tilde{p}_\e \vb w_k^\e\cdot\nabla\phi+(\tilde{\rho}_\e \vb f+\vb g)\cdot(\phi\vb w_k^\e)\,\dx}_{I_5}-\underbrace{\int_D q_k^\e\div(\phi\tilde{\vb u}_\e)\,\dx}_{I_6}.
\end{align*}

Since $\nu_\e:=\tilde{\vb u}_\e$ and $\nu:=\vb u$ fulfill hypothesis (H4) of Theorem~\ref{thm:propertiesCorrectors}, we have
\begin{align*}
I_1\to \mu\,\langle M\vb e_k, \phi\vb u\rangle,
\end{align*}

where $\langle \cdot , \cdot\rangle$ denotes the dual product of $W^{-1,2}(D)$ and $W_0^{1,2}(D)$. Further, by $\tilde{\vb u}_\e\to \vb u$ strongly in $L^2(D)$ and $\nabla\vb w_k^\e\weak 0$ by hypothesis (H3),
\begin{align*}
I_2\to \mu \int_D \nabla \vb u:(\vb e_k\otimes\nabla\phi)\,\dx.
\end{align*}

Because of $\vb w_k^\e\to \vb e_k$ strongly in $L^2(D)$ and \eqref{eq:solenoidal}, we deduce
\begin{align*}
I_3\to 0,\quad I_5\to \int_D p\, \vb e_k\cdot\nabla\phi+(\rho_0 \vb f+\vb g)\cdot(\phi\vb e_k)\,\dx.
\end{align*}

\emph{Step 4:} To show convergence of $I_4$, we proceed as follows. First, since $\vb u_\e=0$ on $\d D_\e$ and $\tilde{\vb u}_\e\weak \vb u$ in $W^{1,2}(D)$, we have $\widetilde{\nabla\vb u_\e}=\nabla\tilde{\vb u}_\e\weak \nabla\vb u$ in $L^2(D)$. Second, as shown above for $\gamma\geq 3$, $\tilde{\rho}_\e\to\rho_0$ strongly in $L^{2\gamma}(D)$ and $\tilde{\vb u}_\e\to \vb u$ strongly in $L^q(D)$ for any $1\leq q<6$, in particular in $L^4(D)$. Together with the strong convergence of $\vb w_k^\e$ in any $L^p(D)$ (see Theorem~\ref{thm:propertiesCorrectors}), in particular in $L^{12}(D)$, we get
\begin{align*}
\tilde{\rho}_\e\tilde{\vb u}_\e\otimes\vb w_k^\e\to \rho_0\vb u\otimes\vb e_k \text{ strongly in } L^2(D).
\end{align*}

This together with $\div(\rho_\e\vb u_\e)=0$ yields
\begin{align*}
I_4&=\int_{D_\e} (\rho_\e\vb u_\e\otimes\vb u_\e):\nabla (\phi \vb w_k^\e)\,\dx=-\int_{D_\e} \rho_\e\vb u_\e\cdot\nabla\vb u_\e\cdot \phi\vb w_k^\e\,\dx=-\int_{D_\e}\phi \nabla\vb u_\e:(\rho_\e\vb u_\e\otimes \vb w_k^\e)\,\dx\\
&=-\int_D \phi \nabla \tilde{\vb u}_\e:(\tilde{\rho}_\e\tilde{\vb u}_\e\otimes\vb w_k^\e)\,\dx\to -\int_D \phi \nabla\vb u:(\rho_0 \vb u\otimes\vb e_k)\,\dx=\int_D (\rho_0\vb u\otimes\vb u):\nabla(\phi\vb e_k)\,\dx.
\end{align*}

In the case $\gamma>3$, one can also proceed by seeing that
\begin{align*}
\tilde{\rho}_\e\tilde{\vb u}_\e\otimes\tilde{\vb u}_\e\to \rho_0\vb u\otimes\vb u \text{ strongly in } L^2(D),
\end{align*}

where we used that $\tilde{\vb u}_\e\to \vb u$ strongly in $L^q(D)$ for $q=4\gamma/(\gamma-1)<6$.\\

\emph{Step 5:} It remains to show convergence of $I_6$. First, recall $B_i^r=B_r(x_i^\e)$. We follow the idea of \cite{SchwarzacherDarcy} and introduce a further splitting of the integral:\\

Let $\psi\in C_c^\infty(B_{1/2}(0))$ be a cut-off function with $\psi=1$ on $B_{1/4}(0)$, define for $x\in B_i^{\e/2}$ the function $\psi_\e^i(x):= \psi((x-x_i^\e)/\e)$, and extend $\psi_\e^i$ by zero to the whole of $D$. Set finally $\psi_\e(x):=\sum\limits_{i:\overline{P_i^\e}\subset D} \psi_\e^i(x)$, where $P_i^\e$ is the cell of size $2\e$ with center $x_i^\e\in (2\e\Z)^3$. Then we have $\psi_\e\in C_c^\infty(\bigcup_i B_i^{\e/2})$ and
\begin{align}\label{def:CutoffPsi}
\psi_\e=1 \text{ in } \bigcup_i B_i^{\e/4},\quad |\nabla\psi_\e|\leq C\e^{-1}.
\end{align}
With this at hand, we write
\begin{align*}
\langle\rho_\e\rangle_\e\cdot I_6 &=\langle\rho_\e\rangle_\e\int_{D_\e} q_k^\e\psi_\e\div(\phi\vb u_\e)\,\dx+\langle\rho_\e\rangle_\e\int_{D_\e} q_k^\e (1-\psi_\e)\phi\div(\vb u_\e)\,\dx\\
&\qquad+\langle\rho_\e\rangle_\e\int_{D_\e} q_k^\e (1-\psi_\e)\vb u_\e\cdot\nabla\phi\,\dx\\
&=:I^1+I^2+I^3.
\end{align*}

Observe that since $\text{supp}~\psi_\e\subset \cup_i B_i^{\e/2}$, the term $I^1$ covers the behavior of $q_k^\e$ ``near'' the holes, whereas $I^2$ and $I^3$ cover the behavior ``far away''. Since $q_k^\e$ and $\psi_\e$ are $(2\e)$-periodic functions and $q_k^\e\psi_\e\in L^2(D)$, we have $q_k^\e\psi_\e\weak 0$ in $L^2(D)/\R$. This together with $\tilde{\vb u}_\e\to \vb u$ strongly in $L^2(D)$ yields
\begin{align*}
|I^3|\to 0.
\end{align*}

For $I^2$, we use the definition of $q_k^\e$ and \eqref{eq:estimateL2_wk} to find
\begin{align*}
|I^2|\leq C \int_{D\setminus\cup_i B_i^{\e/4}} |q_k^\e|\,|\div(\vb u_\e)|\,\dx\overset{\eqref{eq:unifBounds}}{\leq} C\|q_k^\e\|_{L^2(D\setminus \cup_i B_i^{\e/4})}=C\|q_k^\e\|_{L^2(\cup_i B_i^\e\setminus B_i^{\e/4})}\leq C\e\to 0.
\end{align*}

To prove $I^1\to 0$, we write, using $\div(\rho_\e \vb u_\e)=0$,
\begin{align*}
I^1&=\int_{D_\e}\nabla(q_k^\e \psi_\e \phi)\cdot(\rho_\e \vb u_\e)\,\dx-\int_{D_\e}\nabla(q_k^\e \psi_\e \phi)\cdot(\langle\rho_\e\rangle_\e \vb u_\e)\,\dx+\langle\rho_\e\rangle_\e\int_{D_\e} q_k^\e\psi_\e\vb u_\e\cdot\nabla\phi\,\dx\\
&=\int_{D_\e}\nabla(q_k^\e \psi_\e \phi)(\rho_\e-\langle\rho_\e\rangle_\e)\cdot \vb u_\e\,\dx+o(1).
\end{align*}

Here, we used again the periodicity of $q_k^\e$ and $\psi_\e$ to conclude $q_k^\e\psi_\e\weak 0$ in $L^2(D)/\R$. This and the strong convergence of $\tilde{\vb u}_\e$ to $\vb u$ in $L^2(D)$ shows that the last term vanishes in the limit $\e\to 0$. For the remaining integral, we find, recalling $\text{supp}~\psi_\e\subset \cup_i B_i^{\e/2}$ and $C_i^\e=B_i^{\e/2}\setminus T_i^\e$,
\begin{align*}
|I^1|&\leq \|\nabla(q_k^\e\psi_\e\phi)\|_{L^\frac{2\gamma}{\gamma-1}(\cup_i C_i^\e)} \|\rho_\e-\langle\rho_\e\rangle_\e\|_{L^{2\gamma}(D_\e)} \|\vb u_\e\|_{L^2(D_\e)}+o(1)\\
&\leq C\e^\frac{\beta}{\gamma} \|\nabla(q_k^\e\psi_\e\phi)\|_{L^\frac{2\gamma}{\gamma-1}(\cup_i C_i^\e)}+o(1).
\end{align*}

Since $|\nabla\psi_\e|\leq C\e^{-1}$, we have
\begin{align*}
|\nabla(q_k^\e\psi_\e\phi)|\leq C \bigg( |\nabla q_k^\e|+\frac{1}{\e} |q_k^\e|\bigg),
\end{align*}

thus
\begin{align*}
|I^1|\leq C \e^\frac{\beta}{\gamma}\bigg(\|\nabla q_k^\e\|_{L^\frac{2\gamma}{\gamma-1}(\cup_i C_i^\e)}+\frac{1}{\e} \|q_k^\e\|_{L^\frac{2\gamma}{\gamma-1}(\cup_i C_i^\e)}\bigg)+o(1).
\end{align*}

Together with \eqref{eq:estimateLp_wk} and \eqref{eq:estimateNabla_qk} for $p=2\gamma/(\gamma-1)>3/2$, we establish
\begin{align*}
|I^1|\leq C \e^\frac{\beta}{\gamma}\bigg(\e^{-3-\frac{3}{\gamma}}+\e^{-1-\frac3\gamma}\bigg)+o(1)\leq C\e^{-3+\frac{\beta-3}{\gamma}}+o(1)\to 0,
\end{align*}

provided
\begin{align*}
\beta>3\,(\gamma+1).
\end{align*}

To summarize, we have in the limit $\e\to 0$ for all functions $\phi\in\mathcal{D}(D)$
\begin{align*}
\mu\,\langle M\vb e_k,\phi\vb u\rangle-\mu\, \langle\Delta\vb u,\phi\vb e_k\rangle=-\langle\div(\rho_0 \vb u\otimes\vb u),\phi\vb e_k\rangle+\langle \rho_0\vb f+\vb g-\nabla p,\phi\vb e_k\rangle.
\end{align*}

Since $M$ is symmetric, this is
\begin{align*}
\nabla p+\rho_0\vb u\cdot\nabla \vb u -\mu\Delta\vb u+\mu M\vb u=\rho_0\vb f+\vb g \text{ in } \mathcal{D}^\prime(D),
\end{align*}

which is the first equation of \eqref{eq:Brinkman}. This finishes the proof.

\end{proof}

\noindent
\\
{\bf Acknowledgement.} The authors were partially supported by the German Science Foundation DFG in context of
the Emmy Noether Junior Research Group BE 5922/1-1.

\bibliographystyle{amsplain}
\bibliography{Lit}

\providecommand{\bysame}{\leavevmode\hbox to3em{\hrulefill}\thinspace}
\providecommand{\MR}{\relax\ifhmode\unskip\space\fi MR }
\providecommand{\MRhref}[2]{%
  \href{http://www.ams.org/mathscinet-getitem?mr=#1}{#2}
}
\providecommand{\href}[2]{#2}
\begin{thebibliography}{10}

\bibitem{Allaire89}
Gr\'{e}goire Allaire, \emph{Homogenization of the {S}tokes flow in a connected
  porous medium}, Asymptotic Anal. \textbf{2} (1989), no.~3, 203--222.
  \MR{1020348}

\bibitem{Allaire90a}
\bysame, \emph{Homogenization of the {N}avier--{S}tokes equations in open sets
  perforated with tiny holes. {I}. {A}bstract framework, a volume distribution
  of holes}, Arch. Rational Mech. Anal. \textbf{113} (1990), no.~3, 209--259.
  \MR{1079189}

\bibitem{Allaire90b}
\bysame, \emph{Homogenization of the {N}avier--{S}tokes equations in open sets
  perforated with tiny holes. {II}. {N}oncritical sizes of the holes for a
  volume distribution and a surface distribution of holes}, Arch. Rational
  Mech. Anal. \textbf{113} (1990), no.~3, 261--298. \MR{1079190}

\bibitem{BellaOschmann2021a}
Peter Bella and Florian Oschmann, \emph{Inverse of divergence and
  homogenization of compressible {N}avier--{S}tokes equations in randomly
  perforated domains}, arXiv preprint arXiv:2103.04323 (2021).

\bibitem{CioranescuMurat82}
Do{\"{\i}}na Cioranescu and Fran\c{c}ois Murat, \emph{Un terme \'{e}trange venu
  d'ailleurs. {I}}, Nonlinear partial differential equations and their
  applications. {C}oll\`ege de {F}rance {S}eminar, {V}ol. {III}, Res. Notes in
  Math., vol.~70, Pitman, Boston, Mass.-London, 1982, pp.~154--178, 425--426.
  \MR{670272}

\bibitem{DieningFeireislLu}
Lars Diening, Eduard Feireisl, and Yong Lu, \emph{The inverse of the divergence
  operator on perforated domains with applications to homogenization problems
  for the compressible {N}avier--{S}tokes system}, ESAIM: Control, Optimisation
  and Calculus of Variations \textbf{23} (2017), no.~3, 851--868.

\bibitem{FeireislLu}
Eduard Feireisl and Yong Lu, \emph{Homogenization of stationary
  {N}avier--{S}tokes equations in domains with tiny holes}, Journal of
  Mathematical Fluid Mechanics \textbf{17} (2015), no.~2, 381--392.

\bibitem{Galdi2011}
Giovanni~Paolo Galdi, \emph{An introduction to the mathematical theory of the
  {N}avier--{S}tokes equations}, second ed., Springer Monographs in
  Mathematics, Springer, New York, 2011, Steady-state problems. \MR{2808162}

\bibitem{Giunti21}
Arianna Giunti, \emph{Derivation of {D}arcy's law in randomly punctured
  domains}, arXiv preprint arXiv:2101.01046 (2021).

\bibitem{GiuntiHoefer19}
Arianna Giunti and Richard~Matthias H{\"o}fer, \emph{Homogenisation for the
  {S}tokes equations in randomly perforated domains under almost minimal
  assumptions on the size of the holes}, Ann. Inst. H. Poincar\'{e} Anal. Non
  Lin\'{e}aire \textbf{36} (2019), no.~7, 1829--1868. \MR{4020526}

\bibitem{GiuntiHoeferVelazquez18}
Arianna Giunti, Richard~Matthias H{\"o}fer, and Juan J.~L. Vel\'{a}zquez,
  \emph{Homogenization for the {P}oisson equation in randomly perforated
  domains under minimal assumptions on the size of the holes}, Comm. Partial
  Differential Equations \textbf{43} (2018), no.~9, 1377--1412. \MR{3915491}

\bibitem{Hillairet18}
Matthieu Hillairet, \emph{On the homogenization of the {S}tokes problem in a
  perforated domain}, Arch. Ration. Mech. Anal. \textbf{230} (2018), no.~3,
  1179--1228. \MR{3851058}

\bibitem{SchwarzacherDarcy}
Richard~Matthias H{\"o}fer, Karina Kowalczyk, and Sebastian Schwarzacher,
  \emph{Darcy’s law as low {M}ach and homogenization limit of a compressible
  fluid in perforated domains}, Mathematical Models and Methods in Applied
  Sciences \textbf{31} (2021), no.~09, 1787--1819.

\bibitem{LuSchwarzacher}
Yong Lu and Sebastian Schwarzacher, \emph{Homogenization of the compressible
  {N}avier–-{S}tokes equations in domains with very tiny holes}, Journal of
  Differential Equations \textbf{265} (2018), no.~4, 1371 -- 1406.

\bibitem{Masmoudi02}
Nader Masmoudi, \emph{Homogenization of the compressible {N}avier--{S}tokes
  equations in a porous medium}, ESAIM: {C}ontrol, {O}ptimisation and
  {C}alculus of {V}ariations \textbf{8} (2002), 885--906.

\bibitem{Novotny2004}
Antonín Novotn\'y and Ivan Stra\v{s}kraba, \emph{Introduction to the
  {M}athematical {T}heory of {C}ompressible {F}low}, OUP Oxford, New York,
  London, 2004.

\bibitem{Tartar1980}
Luc Tartar, \emph{Incompressible fluid flow in a porous medium -- convergence
  of the homogenization process}, Appendix of Non-homogeneous media and
  vibration theory (1980).

\bibitem{Temam1977navier}
Roger Temam, \emph{{N}avier--{S}tokes equations: theory and numerical
  analysis}, North-Holland Publishing Company, 1977.

\end{thebibliography}

\end{document}